\documentclass[11pt,reqno,a4paper]{amsart}

\usepackage{a4wide}
\usepackage{amssymb}
\usepackage{amsfonts}
\usepackage{enumerate}
\usepackage{amsmath}
\usepackage{bbm}
\usepackage{ stmaryrd }

\newtheorem{theorem}{Theorem}

\newtheorem{definition}[theorem]{Definition}
\newtheorem{example}[theorem]{Example}

\newtheorem{lemma}[theorem]{Lemma}

\newtheorem{proposition}[theorem]{Proposition}

\def\eps{\varepsilon}

\def\N{\mathbb{N}}

\def\R{\mathbb{R}}
\def\P{\mathbb{P}}
\def\E{\mathbb{E}}

\def\F{\mathcal{F}}

\DeclareMathOperator{\var}{var}

\begin{document}

\title[On the shortest distance between orbits]{On the shortest distance between orbits and the longest common substring problem}
\author{Vanessa Barros, Lingmin Liao, Jerome Rousseau}
\address{Vanessa Barros and J\'er\^ome Rousseau, Departamento de Matem\'atica, Universidade Federal da Bahia\\
Av. Ademar de Barros s/n, 40170-110 Salvador, Brazil}
\email{vbarrosoliveira@gmail.com}
\urladdr{https://sites.google.com/site/vbarrosoliveira/home}
\email{jerome.rousseau@ufba.br}
\urladdr{http://www.sd.mat.ufba.br/~jerome.rousseau}

\address{Lingmin Liao, LAMA (Laboratoire d'Analyse et de Math\'ematiques Appliqu\'ees), Universit\'e Paris-Est Cr\'eteil Val de Marne, 61 avenue du G\'en\'eral de Gaulle, 94010 Cr\'eteil Cedex, France }
\email{lingmin.liao@u-pec.fr }
\urladdr{http://perso-math.univ-mlv.fr/users/liao.lingmin/}
\thanks{This work was partially supported by CNPq and FAPESB}

\maketitle

\begin{abstract}
In this paper, we study the behaviour of the shortest distance between orbits and show that under some rapidly mixing conditions, the decay of the shortest distance depends on the correlation dimension. For irrational rotations, we prove a different behaviour depending on the irrational exponent of the angle of the rotation.
For random processes, this problem corresponds to the longest common substring problem. We extend the result of \cite{AW} on sequence matching to $\alpha$-mixing processes with exponential decay. 
\end{abstract}
\section{Introduction}
Motivations to study sequence matching or sequence alignment can be found in various fields of research (e.g. computer science, biology, bioinformatics, geology and linguistics, etc). For instance, to compare two DNA strands, one can be interested in finding the longest common substring, i.e. the longest string of DNA which appears in both strands. Thus, one can measure the level of relationship of the two strands by studying the length of this common substring. For example, for the following two strands
\[ACAATGAGAGGATGACCTTG\]
\[TGACTGTAACTGACACAAGC\]
a longest common substring is ACAA (TGAC is also a longest common substring) and is of length 4 when the total length of the strands is 20.

Other quantities may be of interest in DNA comparison or more generally in sequence alignment and we refer to \cite{RSW,W-book} for more information on the subject. Here we will concentrate on the behaviour of the  length of the longest common substring when the length of the strings grows, more precisely, for two sequences $X$ and $Y$, the behaviour, when $n$ goes to infinity, of
\[M_n(X,Y)=\max\{m:X_{i+k}=Y_{j+k}\textrm{ for $k=1,\dots,m$ and for some $0\leq i,j\leq n-m$}\}.\]
This problem was studied by Arratia and Waterman \cite{AW}, who proved that if $X_1,X_2,\dots$,\\$Y_1,Y_2,\dots$ are i.i.d. such that $\P(X_1=Y_1)=p\in (0,1)$ then
\[\P\left(\lim_{n\rightarrow\infty}\frac{M_n}{\log n}=\frac{2}{-\log p}\right)=1.\]
The same result was also proved for independent irreducible and aperiodic Markov chains on a finite alphabet,  and in this case $p$ is the largest eigenvalue of the matrix $[(p_{ij})^2]$ (where $[p_{ij}]$ is the transition matrix).

In this paper, we generalize Arratia and Waterman's result to $\alpha$-mixing process with exponential decay (or $\psi$-mixing with polynomial decay) and prove that if the R\'enyi entropy $H_2$ exists then
\[\P\left(\lim_{n\rightarrow\infty}\frac{M_n}{\log n}=\frac{2}{H_2}\right)=1.\]

Our theorem applies to both cases of \cite{AW} which are $\alpha$-mixing with exponential decay. Other examples of  $\alpha$-mixing process with exponential decay include Gibbs states of a H\"older-continuous potential \cite{Bowen,Ruelle}. One can see \cite{Bradley} for a nice introduction on strong mixing conditions of stochastic processes (or \cite{brad1} for a more complete version).

Further developments of the work \cite{AW} (e.g. sequences of different lengths, different distributions, more than two sequences, extreme value theory for sequence matching) can be found in \cite{AW1,AGW,AW2,AW3,KO,DKZ,Neu}. We also refer the reader to \cite{pittel,konto,lambert} for related sequence matching problems.

A generalization of the longest common substring problem for dynamical systems is to study the behaviour of the shortest distance between two orbits, that is, for a dynamical system $(X,T,\mu)$, the behaviour, when $n$ goes to infinity, of
\[m_n(x,y)=\min_{i,j=0,\dots,n-1}\left(d(T^ix,T^jy)\right).\]
Indeed, when $X=\mathcal{A}^\N$ for some alphabet $\mathcal{A}$ and $T$ is the shift on $X$, we can consider the distance between two sequences $x,y\in X$ defined by $d(x,y)=e^{-k}$ where $k=\inf\{i\geq0, x_i\neq y_i\}$. 

Then, assuming that $m_n$ is not too small, that is $-\log m_n(x,y)\leq n$ (we will see in Theorem~\ref{thineq} that this condition is satisfied for almost all couples $(x,y)$ if $n$ is large enough), one can observe that almost surely
\[M_n(x,y)\leq -\log m_n(x,y)\leq M_{2n}(x,y).\] 
Thus ${M}_n(x,y)$  and $-\log m_n(x,y)$ have the same asymptotic behaviour.


Even if the shortest distance between two orbits seems to be something natural to define and study, to the best of our knowledge, it has not been done in the literature before. One can observe that this quantity shares some similarities with the correlation sum and the correlation integral of the Grassberger-Procaccia algorithm \cite{gp1,gp2} and the nearest neighbour analysis \cite{cutler}, with the synchronization of coupled map lattices \cite{vaienti1}, with dynamical extremal index \cite{faranda-vaienti}, with the connectivity, proximality and recurrence gauges defined by Boshernitzan and Chaika \cite{BC} and also with logarithm laws and shrinking target properties (see e.g. the survey \cite{athreya}). One can also remark that information on the hitting time (see e.g. \cite{saussol}) can give information on the shortest distance. Indeed, if we define the hitting time of a point $x$ in the ball $B(y,r)$ as $W_r(x,y)=\inf\{k\geq1,T^kx\in B(y,r)\}$ then if $W_r(x,y)\leq n$, we have $m_n(x,y)<r$. 

In this paper, we show that the behaviour of the shortest distance $m_n$ is linked to the correlation dimension of the invariant measure $\mu$, defined (when the limit exists) by
\[{C}_\mu=\underset{r\rightarrow0}{\lim}\frac{\log\int_X\mu\left(B\left(x,r\right)\right)d\mu(x)}{\log r}.\]
More precisely, if the correlation dimension exists, then under some rapid mixing conditions of the system $(X,T,\mu)$, we deduce that for $\mu\otimes\mu$-almost every $(x,y)\in X\times X$,
\[ \underset{n\rightarrow+\infty}{\lim}\frac{\log m_n(x,y)}{-\log n}=\frac{2}{{C}_\mu}.\]
For irrational rotations, we prove that this result does not hold and that the previous limit depends on the irrationality exponent of the angle of the rotation. In the proof, the duality between hitting times and the shortest distance and the result of Kim and Seo \cite{KimSeo} on hitting times for irrational rotations are useful.

Our main results on the shortest distance between orbits and its relation with the correlation dimension are stated in Section~\ref{secshort} and proved in Section~\ref{sec-proof}.  In Section~\ref{sec-lcs}  we state an equivalent formulation of our main theorem (Theorem \ref{thprinc}) for random processes. More precisely, we establish a relation between the longest common substring and the R\'enyi entropy. This result is proved in Section~\ref{sec-discrete}.
The case of irrational rotations is treated in Section~\ref{sec-rot}. We apply our results to multidimensional expanding maps in Section~\ref{secexp}. 
\medskip

\section{Shortest distance between orbits}\label{secshort} 

Let $(X,d)$ be a finite dimensional metric space and $\mathcal{A}$ its Borel $\sigma$-algebra.
Let $(X,\mathcal{A},\mu,T)$ be a measure preserving system (m.p.s.) which means that $T:X\rightarrow X$ is a transformation on $X$ and $\mu$ is a probability measure on $(X,\mathcal{A})$ such that $\mu$ is invariant by $T$, i.e., $\mu(T^{-1}A)=\mu(A)$ for all $A\in\mathcal{A}$.  

We would like to study the behaviour of the shortest distance between two orbits:
\[m_n(x,y)=\min_{i,j=0,\dots,n-1}\left(d(T^ix,T^jy)\right).\]
We will show that the behaviour of $m_n$ as $n\rightarrow\infty$ is linked with the correlation dimension. Before stating the next theorem, we recall the definition of the lower and upper correlation dimensions of $\mu$:
\[\underline{C}_\mu=\underset{r\rightarrow0}{\underline\lim}\frac{\log\int_X\mu\left(B\left(x,r\right)\right)d\mu(x)}{\log r}\qquad\textrm{and}\qquad\overline{C}_\mu=\underset{r\rightarrow0}{\overline\lim}\frac{\log\int_X\mu\left(B\left(x,r\right)\right)d\mu(x)}{\log r}.\]
When the limit exists we will denote the common value of $\underline{C}_\mu$ and $\overline{C}_\mu$ by $C_\mu$. The existence of the correlation dimension and its relation with other dimensions can be found in \cite{Pes, PW, barbaroux}.
\begin{theorem}\label{thineq}

Let $(X,\mathcal{A},\mu,T)$ be a measure preserving system (m.p.s.) such that $\underline{C}_\mu>0$. Then for $\mu\otimes\mu$-almost every $(x,y)\in X\times X$,
\[ \underset{n\rightarrow+\infty}{\overline\lim}\frac{\log m_n(x,y)}{-\log n}\leq\frac{2}{\underline{C}_\mu}.\]

\end{theorem}

Theorem \ref{thineq} is a general result which can be applied to any dynamical system such that $\underline{C}_\mu>0$ and shows us that $m_n$ cannot be too small. If $\underline{C}_\mu=0$, one cannot expect to obtain such information since one can have $m_n(x,y)=0$ on a set of positive measure (for example if the measure $\mu$ is a finite linear combination of Dirac measures). We can also observe that the inequality in Theorem \ref{thineq} can be strict (noting for example the trivial case when $T$ is the identity; a more interesting example, irrational rotations, will be treated in Section~\ref{sec-rot}) but under some natural rapidly mixing conditions we will prove an equality.

\smallskip

We need the following hypotheses. \\
(H1) There exists a Banach space $\mathcal{C}$, such that for all $\psi,\ \phi\in \mathcal{C}$ and for all $n\in\N^*$,  we have
\[\left|\int_X\psi.\phi\circ T^n\, d\mu-\int_X \psi d\mu\int_X\phi d \mu\right|\leq\|\psi\|_\mathcal{C}\|\phi\|_\mathcal{C}\theta_n,\]
with $\theta_n =a^{n}$ ($0\leq a<1$) and where $\|\cdot\|_\mathcal{C}$ is the norm in the Banach space $\mathcal{C}$.

There exist $0<r_0<1$, $c\geq0$ and $\xi \geq0$ such that

\noindent(H2) For any $0<r<r_0$, the function $\psi_1:x\mapsto \mu(B(x,r))$ belongs to the Banach space $\mathcal{C}$ and
\[\|\psi_1\|_\mathcal{C}\leq cr^{-\xi}.\]
(H3) For $\mu$-almost every $y\in X$ and any $0<r<r_0$, the function $\psi_2:x\mapsto \mathbbm{1}_{B(y,r)}(x)$ belongs to the Banach space $\mathcal{C}$ and
\[\|\psi_2\|_\mathcal{C}\leq cr^{-\xi}.\]
We observe that the hypothesis (H3) cannot be satisfied when the Banach space $\mathcal{C}$ is the space of H\"older functions since the characteristic functions are not continuous. We will treat this case separately in Theorem~\ref{thprinc2}.

We will also need some topological information on the space $X$.
\begin{definition}\label{deftight}
A separable metric space $(X,d)$ is called tight if there exist $r_0>0$ and $N_0\in \N$, such that for any $0 < r < r_0$ and any $x\in X$ one can cover $B(x, 2r)$ by at most $N_0$ balls of radius $r$.
\end{definition}
We observe that this is not a very restrictive condition. Indeed, any subset of $\R^n$ with the Euclidian metric is tight and any subset of a Riemannian manifold of bounded curvature is tight (see \cite{GY}). In \cite{GY} it was also proved that if $(X,d)$ admits a doubling measure then it is tight and some examples of spaces which are not tight were given. 

Now we can state our main result.
\begin{theorem}\label{thprinc}
Let $(X,\mathcal{A},\mu,T)$ be a measure preserving system, such that $(X,d)$ is tight, satisfying (H1), (H2), (H3) and such that ${C}_\mu$ exists and is strictly positive. Then for $\mu\otimes\mu$-almost every $(x,y)\in X\times X$,
\[ \underset{n\rightarrow+\infty}{\lim}\frac{\log m_n(x,y)}{-\log n}=\frac{2}{{C}_\mu}.\]
\end{theorem}
Since our hypotheses are similar to the ones in \cite{FFT2, AFV}, it is natural to apply our theorem to the same family of examples. Here, we give a short list of simple examples. In Section~\ref{secexp}, we apply our results to a more interesting family of examples: multidimensional piecewise expanding maps.

\smallskip
Denote by $Leb$ the Lebesgue measure.
\begin{example}
Theorem~\ref{thprinc} can be applied to the following systems:
\begin{enumerate}
\item For $m\in\{2,3,\dots\}$, let $T: [0,1] \rightarrow [0,1]$ be such that $x\mapsto mx \mod 1$ and $\mu=Leb$.
\item Let $T:(0,1]\rightarrow(0,1]$ be such that $T(x)=2^k(x-2^{-k})$ for $x\in(2^{-k},2^{-k+1}]$ and $\mu=Leb$.
\item ($\beta$-transformations) For $\beta>1$, let $T: [0,1] \rightarrow [0,1]$ be such that $x\mapsto \beta x \mod 1$ and $\mu$ be the Parry measure (see  \cite{parry}), which is an absolutely continuous probability measure with density $\rho$ satisfying $1-\frac{1}{\beta}\leq\rho(x)\leq(1-\frac{1}{\beta})^{-1}$ for all $x\in[0,1]$. 
\item (Gauss map) Let $T:(0,1]\rightarrow(0,1]$ be such that $T(x)=\left\{\frac{1}{x}\right\}$ and $d\mu=\frac{1}{\log 2}\frac{dx}{1+x}$.
\end{enumerate}
In these examples it is easy to see that $C_\mu=1$. Moreover, (H1), (H2) and (H3) are satisfied with the Banach space $\mathcal{C}=BV$ of functions having bounded variations (see e.g. \cite{FFT} Section 4.1 and \cite{philipp,philipp2,hof}).
\end{example}

One can observe that Theorem \ref{thprinc} is an immediate consequence of Theorem \ref{thineq} and the next theorem.
\begin{theorem}\label{dsup}

Let $(X,\mathcal{A},\mu,T)$ be a measure preserving system, such that $\underline{C}_\mu>0$ and such that $(X,d)$ is tight, satisfying (H1), (H2) and (H3). Then for $\mu\otimes\mu$-almost every $(x,y)\in X\times X$,
\[ \underset{n\rightarrow+\infty}{\underline\lim}\frac{\log m_n(x,y)}{-\log n}\geq\frac{2}{\overline{C}_\mu}.\]

\end{theorem}

When the Banach space $\mathcal{C}$ is the space of H\"older functions $\mathcal{H}^\alpha(X,\R)$ we can adapt our proof and do not need to assume (H3). Moreover, (H2) can be replaced by a condition on the measure of an annulus:

(HA) There exist $r_0>0$, $\xi\geq0$ and $\beta>0$ such that for every $x\in X$ and any $r_0>r>\rho>0$,
\[\mu(B(x,r+\rho)\backslash B(x,r-\rho))\leq r^{-\xi}\rho^\beta.\]
In fact, we will show in the proof of the next theorem that (HA) implies (H2). 
Analogous conditions to (HA) have already appeared in the literature (e.g. \cite{GHN, saussol, chazottes, haydn-w}) but in a local version. Here, we need a stronger global version. Nevertheless, one can easily observe that this assumption is still satisfied if the measure is Lebesgue or absolutely continuous with respect to Lebesgue with a bounded density. 

\begin{theorem}\label{thprinc2}

Let $(X,\mathcal{A},\mu,T)$ be a measure preserving system, such that $\underline{C}_\mu>0$ and such that $(X,d)$ is tight, satisfying (H1) with $\mathcal{C}=\mathcal{H}^\alpha(X,\R)$ and (HA) or (H2). Then for $\mu\otimes\mu$-almost every $(x,y)\in X\times X$,
\[ \underset{n\rightarrow+\infty}{\underline\lim}\frac{\log m_n(x,y)}{-\log n}\geq\frac{2}{\overline{C}_\mu}.\]
\end{theorem}

Here are some interesting examples
where Theorem \ref{thprinc2} applies: planar dispersing billiard maps (with finite and infinite horizon) and Lorenz maps (see \cite{GHN} Section 4 and the references therein), expanding maps of the interval with a Gibbs measure associated to a H\"older potential (see \cite{saussol} where (HA) is proved in Lemma 44) and $C^2$ endomorphism (of a $d$-dimensional compact Riemannian manifold) admitting a Young tower with exponential tail (see \cite{FFT10} Section 6 and \cite{collet}).

\medskip

\section{Longest common substring problem}\label{sec-lcs}
As explained in the introduction, finding the shortest distance between two orbits corresponds, when working with symbolic dynamical systems, to a sequence matching problem: finding the size of the longest common substrings between two sequences.

We will consider the symbolic dynamical system $(\Omega,\P,\sigma)$ where $\Omega=\mathcal{A}^\N$ for some alphabet $\mathcal{A}$, $\sigma$ is the (left) shift on $\Omega$ and $\P$ is an invariant probability measure. For two sequences $x,y\in\Omega$, we are interested in the behaviour of
\[M_n(x,y)=\max\{m:x_{i+k}=y_{j+k}\textrm{ for $k=1,\dots,m$ and for some $1\leq i,j\leq n-m$}\}.\]
We will show that the behaviour of $M_n$ is linked with the R\'enyi entropy of the system. 

 For $y\in \Omega$ we denote by $C_n(y)=\{z \in \Omega :z_i = y_i\text{ for all } 
0\le i\le n-1\} $ the  \emph{$n$-cylinder} containing $y$. Set $\F_0^n$ as the sigma-algebra over $\Omega$ 
generated by all $n$-cylinders.

We define the lower and upper R\'enyi entropy as the following limits:
\[\underline{H}_2=\underset{k\rightarrow+\infty}{\underline\lim}\frac{\log\sum \P(C_k)^2}{-k}\qquad\textrm{and}\qquad\overline{H}_2=\underset{k\rightarrow+\infty}{\overline\lim}\frac{\log\sum \P(C_k)^2}{-k},\]
where the sums are taken over all $k$-cylinders. When the limit exists, we will denote it $H_2$.

The existence of the R\'enyi entropy has been proved for Bernoulli and Markov measures, Gibbs states of a H\"older-continuous potential, weakly $\psi$-mixing processes \cite{haydn-vaienti} and recently for $\psi_g$-regular processes \cite{abadi-car}.

We say that our system is {\it $\alpha$-mixing} if there exists
a function $\alpha:\N \rightarrow\R$ satisfying $\alpha(g)\to0$ when $g\to+\infty$ and such that
for all $m,n \in \N$, $A\in\F_0^n$ and $B\in \F_0^{m}$:
\[
\left|\P(A\cap\sigma^{-g-n}B) -\P(A)\P(B)\right|\le \alpha(g).
\]
It is said to be  {\it $\alpha$-mixing with an exponential decay} if the function $\alpha(g)$ decreases exponentially fast to $0$.

We say that our system is {\it $\psi$-mixing} if there exists
a function $\psi:\N \rightarrow\R$ satisfying $\psi(g)\to 0$ when $g\to+\infty$ and such that
for all $m,n \in \N$, $A\in\F_0^n$ and $B\in \F_0^{m}$:
\[
\left|\P(A\cap\sigma^{-g-n}B) -\P(A)\P(B)\right|\le \psi(g)\P(A)\P(B).
\]
Now we are ready to state our next result.

\begin{theorem}\label{seqmat}

If $\underline{H}_2>0$, then for $\P\otimes\P$-almost every $(x,y)\in \Omega\times \Omega$,
\begin{equation}\label{eqren1}
 \underset{n\rightarrow+\infty}{\overline\lim}\frac{M_n(x,y)}{\log n}\leq\frac{2}{\underline{H}_2}.\end{equation}
Moreover, if the system is $\alpha$-mixing with an exponential decay or if the system is $\psi$-mixing with $\psi(g)=g^{-a}$ for some $a>0$ then, for $\P\otimes\P$-almost every $(x,y)\in \Omega\times \Omega$,
\begin{equation}\label{eqren2} \underset{n\rightarrow+\infty}{\underline\lim}\frac{M_n(x,y)}{\log n}\geq\frac{2}{\overline{H}_2}.\end{equation}
Furthermore, if the R\'enyi entropy exists, then for $\P\otimes\P$-almost every $(x,y)\in \Omega\times \Omega$,
\[ \underset{n\rightarrow+\infty}{\lim}\frac{M_n(x,y)}{\log n}=\frac{2}{H_2}.\]

\end{theorem}

Remark that Theorem \ref{seqmat} generalizes the results in \cite{AW} since the processes treated there (i.i.d. and independent irreducible and aperiodic Markov chains on a finite alphabet) are $\alpha$-mixing with an exponential decay and their R\'enyi entropies exist. Moreover, in \cite{AW} the authors used a different proof for each case, while here we present a single and simpler proof. Our proof 
which will be presented in Section~\ref{sec-discrete} 
is an adaptation to symbolic dynamical systems of the results presented in Section~\ref{secshort}. 

One can apply our results to the following examples (which cannot be obtained from \cite{AW}).

\begin{example}[Gibbs states]
Gibbs states of a H\"older-continuous potential $\phi$ are $\psi$-mixing with an exponential decay \cite{Bowen, Ruelle}. Moreover, the R\'enyi entropy exists and $H_2=2P(\phi)-P(2\phi)$ where $P(\phi)$ is the pressure of the potential $\phi$ \cite{haydn-vaienti}.
\end{example}
\begin{example}[Renewal process]
Let $0<q_i<1$ for any $i\in\N$. Consider the Markov chain $(Y_n)_n$ with the following transition probabilities
\[Q_Y(i,j)=\left\{\begin{array}{ll} q_i & \textrm{if }j=0\\
1-q_i & \textrm{if }j=i+1\\
0 & \textrm{otherwise}\end{array}\right.\]
for any $i\in\N , j\in\N$.

Let $(X_n)_n$ be the process defined by $X_n=1$ when $Y_n=0$ and $X_n=0$ when $Y_n\neq 0$ and let $\P$ be its stationary measure. This process is called a binary renewal process. 

Assuming that there exists $a\in(0,1/2)$ such that $a\leq q_i\leq 1-a$ for any $i\in\N$, then this process is $\alpha$-mixing with exponential decay \cite{abadi-car-gallo}, thus we have the inequalities \eqref{eqren1} and \eqref{eqren2}. However, we observe that the existence of the R\'enyi entropy in this case is not known. 
\end{example}

\medskip

\section{Irrational rotations}\label{sec-rot}

In this section we consider the irrational rotations. For $\theta\in \mathbb{R}\setminus \mathbb{Q}$, let $T_\theta$ be the irrational rotation on the unit circle $\mathbb{T}= \mathbb{R}/\mathbb{Z}$  defined by
\[
T_\theta x= x+ \theta. 
\]
 Then for any $n\in\mathbb{Z}$, we have $T_\theta^n x=x+n\theta$ and the shortest distance becomes
\[
m_n(x,y)=\min_{-n\leq j \leq n} \|(x-y)+j\theta\|.
\]
The limit behavior of $m_n(x,y)$ is thus linked to the inhomogeneous Diophantine approximation.

Let 
\[
\eta=\eta(\theta):=\sup\{\beta\geq 1: \liminf_{j\to\infty} j^\beta\|j\theta\|=0\}
\]
be the irrationality exponent of $\theta$.
Now we will show that the result of Theorem \ref{thprinc} does not hold for $T_\theta$.

\begin{theorem}
For $n\in \mathbb{N}$ and $(x,y)\in \mathbb{T}^2$, let $m_n(x,y)$ be the shortest distance between the orbits of $x$ and $y$ defined as above. Then for Lebesgue almost all $(x,y)\in \mathbb{T}^2$, we have
\begin{align*}
\liminf_{n\to\infty} {\log m_n(x,y) \over -\log n } = {1 \over \eta} \quad \text{and} \quad
 \limsup_{n\to\infty} {\log m_n(x,y) \over -\log n }= 1.
\end{align*}
\end{theorem}

\begin{proof}

Let 
\[
W_{B(y,r)} (x):=\inf\{n\geq 1: T_\theta^n(x)\in B(y,r)\}
\]
be the waiting time for $x\in \mathbb{T}$ entering the ball $B(y,r)$ of center $y\in \mathbb{T}$ and radius $r>0$.  Kim and Seo (\cite{KimSeo}) proved that for almost all $x$ and $y$ in $\mathbb{T}$,  
\begin{equation}\label{KimS}
\liminf_{r\to 0} {\log W_{B(y,r)} (x) \over -\log r}=1 \quad \text{and} \quad \limsup_{r\to 0} {\log W_{B(y,r)} (x) \over -\log r}=\eta.
\end{equation}
Let us denote $\widetilde W_{B(y,r)} (x)$ the waiting time under the action of the irrational rotation of angle $-\theta$. Since $\eta(\theta)=\eta(-\theta)$, \eqref{KimS} is also satisfied for $\widetilde W_{B(y,r)} (x)$.

By definition, for the time $k=W_{B(y,r)} (x)$, $T_\theta^kx$ firstly enters the ball $B(y,r)$. So, we have $m_n(x,y)< r$, when $W_{B(y,r)} (x)\leq n$. On the other hand,  when $ W_{B(y,r)} (x)>n$ and $\widetilde W_{B(y,r)} (x)>n$,  we have $m_n(x,y)> r$.

By the first equality of (\ref{KimS}), there is a sequence $(r_k)_k$ tending to $0$ such that
\[\lim_{k\to \infty} {\log W_{B(y,r_k)} (x) \over -\log r_k}=1.\]
 Let $n_k=W_{B(y,r_k)} (x)$. Then $m_{n_k}(x,y)< r_k$.  Thus,
\begin{equation}\label{eta}
 \limsup_{n\to\infty} {\log m_n(x,y) \over -\log n } \geq \limsup_{k\to\infty} {\log m_{n_k}(x,y) \over -\log n_k } \geq \lim_{k\to\infty}{\log r_k \over -\log W_{B(y,r_k)} (x)}=1.
\end{equation}

Using again the first equality of (\ref{KimS}), for any $0<\epsilon<1$,  we have
\[
W_{B(y,r)} (x)>  (1/r)^{1-\epsilon} \qquad\textrm{and}\qquad \widetilde W_{B(y,r)} (x)>  (1/r)^{1-\epsilon}, 
\]
provided $r>0$ is small enough.

Therefore, taking $0<\epsilon<1$ and defining $r=n^{-1/(1-\epsilon)}$ for $n\gg1$, we get 

\[
W_{B(y,r)} (x)> n \qquad\textrm{and}\qquad \widetilde W_{B(y,r)} (x)> n,
\]
which implies 
\[
m_n(x,y) > r= n^{-1/(1-\epsilon)}.
\]
Thus
\[
\limsup_{n\to\infty} {\log m_n(x,y) \over -\log n } < {1 \over 1-\epsilon}.
\]
By the arbitrariness of $\epsilon>0$, we  have 
\begin{equation}\label{eta1}
\limsup_{n\to\infty} {\log m_n(x,y) \over -\log n } \leq 1. 
\end{equation}

From the inequalities \eqref{eta} and \eqref{eta1} we get the second part of the theorem.

By the same arguments and the second equality of (\ref{KimS}), we deduce that for any $\epsilon>0$
\[
W_{B(y,r)} (x)\leq  (1/r)^{\eta+\epsilon},
\]
provided $r>0$ small enough.

Hence, defining $r=n^{-1/(\eta+\epsilon)},\ n\gg1$, we have 
$
W_{B(y,r)} (x)\leq n,
$
which implies 
$$
m_n(x,y) \leq r= n^{-1/(\eta+\epsilon)}.
$$
Thus
\[
\liminf_{n\to\infty} {\log m_n(x,y) \over -\log n } \geq {1 \over \eta+\epsilon}.
\]
By the arbitrariness of $\epsilon>0$, we then obtain
\begin{equation}\label{eta2}
 \liminf_{n\to\infty} {\log m_n(x,y) \over -\log n } \geq {1 \over \eta}.
\end{equation}

From the inequalities \eqref{eta1} and \eqref{eta2} we see that for $\eta=1$ the following result holds
$$\lim_{n\to\infty} {\log m_n(x,y) \over -\log n } =1. $$
Thus, from now on we can suppose $\eta >1$ and it only remains to show 
\begin{align}\label{upper_liminf}
\liminf_{n\to\infty} {\log m_n(x,y) \over -\log n } \leq {1 \over \eta}.
\end{align}

We remark that the results of Kim and Seo (\cite{KimSeo}) are not applicable for proving (\ref{upper_liminf}), so we will give a direct proof.

Let $q_k=q_k(\theta)$ be the denominators of the $k$-th convergent of the continued fraction of $\theta$. 
Then (see for example Khintchine's book \cite{Khinchin})
\begin{equation}\label{qn-estimate}
{1\over 2q_{n+1}} <  {1\over q_{n+1}+q_n}< \|q_{n}\theta\| \leq {1 \over q_{n+1}}.
\end{equation}
By the theorem of best approximation (e.g. \cite{RS}), we have 
\begin{equation}
\eta(\theta) =\limsup_{n\to\infty} {\log q_{n+1}\over \log q_n}. \label{def-w}
\end{equation}

First note that since
$$m_n(x,y)=m_n(0,x-y),$$
it is enough to show that for almost all $y\in [0,1]$ 
\begin{align*}
\liminf_{n\to\infty} {\log m_n(0,y) \over -\log n } \leq {1 \over \eta}.
\end{align*}

Second, we consider the following function 
\[f:y\in [0,1]\mapsto \liminf_{n\to\infty} {\log m_n(0,y) \over -\log n }.\] 
Observing $$m_{n+1}(0,y)\leq m_n(0,T_\theta y)=\min_{-n+1\leq j \leq n+1} \|y+j\theta\|\leq m_{n-1}(0,y),$$
we see that $f$ is a $T_\theta$-invariant function. By the ergodicity of the Lebesgue measure with respect to the irrational rotations, we conclude that  $f$ is a constant for almost all $y\in[0,1]$. Since we have already proved that
\begin{align*}
\liminf_{n\to\infty} {\log m_n(0,y) \over -\log n } \geq {1 \over \eta},
\end{align*}
we  only need to show that for any $\delta>0$, the set 
\[ 
E_\delta:=\left\{y\in [0,1]: \text{there exists a sequence } \{N_k\}_k \ \text{such that $\forall k$, }  {\log m_{N_k}(0,y) \over \log N_k} \leq {1 \over \eta}+\delta\right\}
\]
has positive Lebesgue measure.
By the definition of $m_n(x,y)$, we can rewrite $E_\delta$ as
\[ 
E_\delta=\left\{y\in [0,1]: \exists \ \{ N_k \}_k  \ \text{s.t. }  \forall\  k, \ \forall j=-N_k+1,...,N_k-1, \ \| j\theta-y \|\geq {1 \over N_k^{{1 \over \eta}+\delta}}\right\}.
\]

Let
$\tau={1 \over \eta}+\delta$ and take $0<\epsilon<{\delta \eta^2 \over 1+\delta\eta}$. From the identity \eqref{def-w} there exists a subsequence $\{n_k\}_k$ such that 
$$q_{n_k+1} \geq q_{n_k}^{\eta-{\epsilon\over 2}} \geq q_{n_k}^{\eta-\epsilon}>q_{n_k}^{\eta \over 1+\delta\eta}, \ \text{ since $q_k \geq 1$. }$$
Without loss of generality, we still write $\{k\}$ this subsequence.

Take $N_k=\lceil q_k^{\eta-\epsilon} \rceil$ and define the following decreasing sequence of sets
\[
E_{\delta,k}:=\left\{y\in [0,1]:  \forall j=-N_k+1,...,N_k-1, \ \| j\theta-y \|\geq {1 \over N_k^{{\tau}}}\right\}.
\]
Since $E_{\delta}=\cap_{k\geq k_0} E_{\delta,k}$ for any  $k_0\in \mathbb{N}$,
we only need to show that  $\cup_{k\geq k_0} E_{\delta,k}^c$ has Lebesgue measure strictly less than $1$ for some $k_0\in \mathbb{N}$.

Now observing that $
E_{\delta,k}^c=\bigcup_{-N_k< j < N_k} B\big(j\theta, {1\over N_k^\tau}\big),
$
we only have to estimate the Lebesgue measure of the following set
$$\bigcup_{k \geq k_0} \bigcup_{-N_k< j < N_k} B\big(j\theta, {1\over N_k^\tau}\big).$$

Let us first consider the union 
$
\bigcup_{0\leq j < N_k} B\big(j\theta, {1\over N_k^\tau}\big).
$

By the definition of $N_k$ and the facts $\eta-\epsilon>1$ and that $\{q_k \}\subset \N$ is increasing we have $$q_k\leq N_k\leq q_{k+1}.$$
Therefore
\begin{align}\label{eq-union}
\bigcup_{0\leq j < N_k} B\big(j\theta, {1\over N_k^\tau}\big) \subset \bigcup_{i=0}^{q_k-1} \bigcup_{j=0}^{\lceil N_k/q_k\rceil} B\Big((i+jq_k)\theta, {1\over N_k^\tau}\Big).
\end{align}
Moreover, we have
 \begin{align}\label{N_k_tau_upper}
 {1\over N_k^\tau} \leq {1 \over q_k^{(\eta-\epsilon)\tau}}= {1 \over q_k^{(\eta-\epsilon)({1 \over \eta}+\delta)}} =  {1 \over q_k^{1+\delta\eta-\epsilon({1 \over \eta}+\delta)}} ={1 \over q_k^{1+\epsilon_1}}. 
\end{align}
  where $\epsilon_1:=\delta\eta-\epsilon({1 \over \eta}+\delta)>0$.
  
To estimate the measure of the following set
\begin{align*}
\bigcup_{j=0}^{\lceil N_k/q_k\rceil} B\Big((i+jq_k)\theta, {1\over N_k^\tau}\Big),
\end{align*} 
one can observe that the distance between two consecutive centers of the balls in the union is $$\| (i+(j+1)q_k)\theta-(i+jq_k)\theta\|=\|q_k\theta\|.$$ Thus
$$Leb\left(\bigcup_{j=0}^{\lceil N_k/q_k\rceil} B\Big((i+jq_k)\theta, {1\over N_k^\tau}\Big)\right)\leq \left({N_k \over q_k}+2\right)\cdot \|q_k\theta\|+{2 \over N_k^\tau} .$$

From the inequalities \eqref{qn-estimate} and \eqref{N_k_tau_upper} we have
\begin{align*}\label{}
 \left({N_k \over q_k}+2\right)\cdot \|q_k\theta\|+{2 \over N_k^\tau}  \leq & \left({q_k^{\eta-\epsilon}+1 \over q_k}+2\right)\cdot {1 \over q_{k+1}} +{2 \over N_k^\tau} 
 \leq {3q_k^{\eta-\epsilon} \over q_k} \cdot {1 \over q_{k}^{\eta-{\epsilon \over 2}}} + {2\over N_k^\tau}\\
 \leq & {3 \over  q_{k}^{1+{\epsilon \over 2}}} + {2\over q_k^{1+\epsilon_1}}
 \leq {5 \over q_k^{1+\epsilon_2}}, 
\end{align*} 
 where $0<\epsilon_2\leq \min\{\epsilon/2,\epsilon_1\}$. 
 
Let $\beta>1$. 
 For $k$ large enough we obtain
 \begin{align*}
Leb\left(\bigcup_{0\leq j < N_k} B\big(j\theta, {1\over N_k^\tau}\big) \right)\leq q_k \cdot {5 \over q_k^{1+\epsilon_2}}={5 \over q_k^{\epsilon_2}} \leq 5 \beta^{-\epsilon_2 k},
\end{align*}
where the last inequality comes from the assumption $\eta(\theta)>1$.

By symmetry, we also deduce that 
 \begin{align*}
Leb\left(\bigcup_{-N_k< j \leq 0} B\big(j\theta, {1\over N_k^\tau}\big) \right)\leq 5 \beta^{-\epsilon_2 k}.
\end{align*}

Therefore
\begin{align*}
Leb\left(\bigcup_{-N_k< j < N_k} B\big(j\theta, {1\over N_k^\tau}\big) \right)\leq 10 \beta^{-\epsilon_2 k}.
\end{align*}

Note that there exists $k_0\geq 1$, such that
\[
10\sum_{k=k_0}^\infty \beta^{-\epsilon_2 k} = {10\beta^{-\epsilon_2k_0} \over 1- \beta^{-\epsilon_2}}<1.
\]
Thus 
\[
\bigcup_{k\geq k_0} E_{\delta k}^c = \bigcup_{k\geq k_0}\bigcup_{-N_k< j < N_k} B\big(j\theta, {1\over N_k^\tau}\big)
\]
has Lebesgue measure strictly less than $1$. Therefore, (\ref{upper_liminf}) follows.

\smallskip
Hence, finally, we conclude that for almost all $x$ and $y$, 
\[
\liminf_{n\to\infty} {\log m_n(x,y) \over -\log n } = {1 \over \eta}, \quad \text{and} \quad
 \limsup_{n\to\infty} {\log m_n(x,y) \over -\log n }= 1.
\]
\end{proof}


\section{Multidimensional piecewise expanding maps}\label{secexp}
In this section, we will apply our main result to a family of maps defined by Saussol \cite{Saussol1}: multidimensional piecewise uniformly expanding maps. It was observed in \cite{AFLV11} that these maps generalize Markov maps which also contain one-dimensional piecewise uniformly expanding maps.

Let $N\geq 1$ be an integer. We will work in the Euclidean space $\mathbb{R}^N$. We denote by $B_\epsilon(x)$ the ball with center $x$ and radius $\epsilon$. For a set $E\subset\mathbb{R}^N$, we write 
$$B_\epsilon(E):=\{y\in\mathbb{R}^N: \sup_{x\in E} |x-y|\leq \epsilon\}.$$

\begin{definition}[Multidimensional piecewise expanding systems]
Let $X$ be a compact subset of $\R^N$ with $\overline{X^\circ}=X$ and $T:X \rightarrow X$. The system $(X,T)$ is a multidimensional piecewise expanding system if there exists a family of at most countably many disjoint open sets $U_i \subset X$ and $V_i$ such that $\overline{U_i} \subset V_i$ and maps $T_i:V_i  \rightarrow \R^N$ satisfying for some $0< \alpha \leq1$, for some small enough $ \epsilon_0>0$, and for all $i$:
\begin{enumerate}
\item $T\vert_{U_i}=T_i\vert_{ U_i} $ and $B_{\epsilon_0}(TU_i)\subset T_i(V_i) $;
\item $T_i \in C^1(V_i),T_i \text{ is injective and }  T_i^{-1}\in C^1(T_iV_i) .$ Moreover, there exists a constant $c$, such that for all $\epsilon \leq \epsilon_0, z \in T_iV_i$ and $x,y \in B_{\epsilon}(z)\cap T_iV_i$ we have 
$$|\det D_xT_i^{-1}-\det D_yT_i^{-1}| \leq c \epsilon^{\alpha} |\det D_zT_i^{-1}|;$$
\item $Leb(X\setminus \bigcup_i U_i)=0$;
\item there exists $s=s(T)<1$ such that for all $u,v \in TV_i$ with $d(u,v) \leq \epsilon_0$ we have $d(T_i^{-1}u,T_i^{-1}v)\leq s d(u,v)$;
\item let $G(\epsilon,\epsilon_0):=\sup_x  G(x,\epsilon,\epsilon_0)$ where
$$G(x,\epsilon,\epsilon_0)=\sum_i\frac{Leb(T_i^{-1}B_{\epsilon}(\partial TU_i)\cap B_{(1-s)\epsilon_0}(x))}{m(B_{(1-s)\epsilon_0}(x))},$$
then the number $\eta=\eta(\delta):=s^\alpha +2 \sup_{\epsilon \leq \delta} \frac{G(\epsilon)}{\epsilon^\alpha}\delta^\alpha$ satisfies $\sup_{\delta \leq \epsilon_0} \eta(\delta)<1.$
\end{enumerate}
\end{definition}
We will prove that the multidimensional piecewise expanding systems satisfy the conditions of Theorem~\ref{thprinc}. 
\begin{proposition}
Let $(X,T)$ be a topologically mixing multidimensional piecewise expanding map and $\mu$ be its absolutely continuous invariant probability measure. If the density of $\mu$ is bounded away from zero, then for $\mu\otimes\mu$-almost every $(x,y)\in X\times X$,
\[ \underset{n\rightarrow+\infty}{\lim}\frac{\log m_n(x,y)}{-\log n}=\frac{2}{N}.\]

\end{proposition}

First of all, we define the Banach space involved in the mixing conditions. Let $\Gamma \subset X$ be a Borel set. We define the oscillation of $\varphi \in L^1(Leb)$ over $\Gamma$ as

$$osc(\varphi,\Gamma)= \underset{ \Gamma}{ \textrm{ess-sup} } (\varphi)- \underset{ \ \Gamma}{ \textrm{ess-inf} } (\varphi).$$

Now, given real numbers $0<\alpha \leq 1$ and $0<\epsilon_0<1$ consider the following $\alpha$-seminorm
$$|\varphi|_\alpha=\underset{ 0<\epsilon\leq\epsilon_0}{ \sup \  } \epsilon^{ -\alpha} \int_{ X}osc(\varphi,B_\epsilon(x))dx.$$
We observe (see \cite{Saussol1}) that $X\ni x \mapsto osc(\varphi,B_\epsilon(x))$ is a measurable function and 
$$\text{supp} (\text{osc}(\varphi,B_\epsilon(x)))\subset B_\epsilon(\text{supp } \varphi).$$ 

Let $V_{\alpha}$ be the space of $L^1(Leb)-$functions such that $|\varphi|_\alpha<\infty$ endowed with the norm
$$\|\varphi\|_{\alpha}=\|\varphi\|_{L^1(Leb)}+|\varphi|_\alpha.$$
Then
$(V_{\alpha},\|\cdot\|_{\alpha})$ is a Banach space which does not depend on the choice of $\epsilon_0 $  and
 $V_{\alpha} \subset L^\infty $  (see \cite{Saussol1}).

Saussol (\cite{Saussol1}) proved that for a piecewise expanding map $T:X \longrightarrow X$, where $X\subset \R^N$ is a compact set, there exists an absolutely continuous invariant probability measure $\mu$ which enjoys exponential  decay of correlations against $L^1$ observables on $V_{\alpha}$. More precisely, for all $\psi\in V_{\alpha}$, for all $\phi\in L^1(\mu)$ and for all $n\in\N^*$,  we have
\[\left|\int_X\psi.\phi\circ T^n\, d\mu-\int_X \psi d\mu\int_X\phi d \mu\right|\leq\|\psi\|_{\alpha}\|\phi\|_1\theta_n,\]
with $\theta_n =a^{n}$ ($0\leq a<1$). This means that the system $(X,T,\mu)$ satisfies the condition (H1) with  $\mathcal{C}=V_\alpha$.

 It remains to show that the system also satisfies the conditions (H2) and (H3) (with $r_0=\epsilon_0$).
To this end, we need to estimate the norms $\|\psi_1\|_{\alpha}$ and $\|\psi_2\|_{\alpha}$, where $\psi_1$ and $\psi_2$ are the functions defined in (H2) and (H3). Since $\psi_1$ and $\psi_2$ are both in $L^1(Leb)$ we just need to estimate their $\alpha$-seminorms.

From the above observation we notice that 
 \begin{align*}
  \textrm{supp} \ osc(\psi_j,B_\epsilon(\cdot)) \subset X_\epsilon,\ j=1,2,
  \end{align*}
 where $X_\epsilon=\{ x\in \R^N, d(x,X)\leq \epsilon\} $ is a compact set.
Therefore
$$|\psi_j|_\alpha=\underset{ 0<\epsilon\leq\epsilon_0}{ \sup \  } \epsilon^{ -\alpha} \int_{ X_\epsilon}osc(\psi_j,B_\epsilon(x))dx,\ j=1,2.$$ 
 

To estimate $|\psi_j|_\alpha,\ j=1,2$ we define  $$S_j^\epsilon:= \epsilon^{ -\alpha} \int_{ X_\epsilon}osc(\psi_j,B_\epsilon(x))dx,$$
and prove that  $S_j^\epsilon$ is bounded from above by
$C_j\epsilon_0^{ 1-\alpha}$, for some $ C_j>0,\ j=1,2. $

Let us start with $S_1^\epsilon=\epsilon^{ -\alpha} \int_{ X_\epsilon}osc(\mu(B(\cdot,r),B_\epsilon(x))dx$.
Suppose $r\leq \epsilon$. Since $\mu$ is absolutely continuous and its density is bounded away from zero, we can write
\begin{align}\label{ACpsi1}
\psi_1(y)=\mu(B(y,r))=\int_{B(y,r)} h(z)dz,
\end{align}
where the density $h$ belongs to  $V_{\alpha}\subset L^\infty$. It means that 
 $0< c_1\leq h \leq c_2$ for some constants $c_1$ and $c_2.$

By  \eqref{ACpsi1} we have
\begin{align*}
 osc(\psi_1,B_\epsilon(x))= \underset{ y\in B(x,\epsilon)\cap X }{ \textrm{ess-sup} } \int_{B(y,r)} h(z)dz- \underset{ \tilde y\in B(x,\epsilon) \cap X}{ \textrm{ess-inf} } \int_{B(\tilde y,r)} h(z)dz.
 \end{align*}
Therefore $S^1_ \epsilon$ becomes
$$S^1_ \epsilon= \epsilon^{ -\alpha} \int_{ X_\epsilon}\left(\underset{ y\in B(x,\epsilon) \cap X}{ \textrm{ess-sup} } \int_{B(y,r)} h(z)dz- \underset{ \tilde y\in B(x,\epsilon) \cap X}{ \textrm{ess-inf} } \int_{B(\tilde y,r)} h(z)dz
\right)dx.$$
Since $c_1 \leq h \leq c_2$, 
\begin{align*}
S_1^\epsilon& \leq  \epsilon^{ -\alpha} \int_{ X_\epsilon}\left(\underset{ y\in B(x,\epsilon) \cap X}{ \textrm{ess-sup} } \int_{B(y,r)} c_2dz- \underset{ \tilde y\in B(x,\epsilon) \cap X}{ \textrm{ess-inf} } \int_{B(\tilde y,r)} c_1dz
\right)dx\\
&\leq C_0\epsilon^{ -\alpha} \int_{ X_\epsilon} (c_2-c_1) r^Ndx \leq   C_0(c_2-c_1)\epsilon^{ -\alpha+N} Leb(X_{\epsilon}),
 \end{align*}
 where $C_0$ is the Lebesgue measure of the unit ball in $\mathbb{R}^N$.
Using the facts that $X_\epsilon\subset X_{\epsilon_0}$ and that $X_{\epsilon_0}$ is compact, we have

\begin{equation}\label{sup1}
S_1^\epsilon \leq   (c_2-c_1)\epsilon^{ -\alpha+N} Leb(X_{\epsilon_0})\leq C \epsilon_0^{ -\alpha+N}.
\end{equation}

Now suppose $r>\epsilon$. Then for each $y\in B(x,\epsilon)$  we have  $$ B(x,r-\epsilon)\subset B(y,r)\subset B(x,r+\epsilon).$$ Therefore
\begin{align}\label{a}
 \ osc(\psi_1,B_\epsilon(x))\leq \int_{B(x,r+\epsilon)} h(z)dz-\int_{B(x,r-\epsilon)} h(z)dz= \int_{D} h(z)dz
 \leq \|h\|_{\infty}Leb(D),
\end{align}
where $D=B(x,r+\epsilon)\setminus B(x,r-\epsilon)$. It is easy to see that
\begin{eqnarray}\label{b}
Leb(D)\leq 2C_0\epsilon  \sum_{k=0}^{N-1}\binom{N}{k}
\leq 2^{N+1} C_0\epsilon .
\end{eqnarray}
From the inequalities \eqref{a} and \eqref{b} we deduce that
\begin{eqnarray}\label{c}
S_1^\epsilon\leq\epsilon^{-\alpha}\|h\|_{\infty}Leb(D)Leb(X_\epsilon)
\leq 2^N \epsilon_0^{1-\alpha}\|h\|_{\infty}Leb(X_{\epsilon_0}) \leq C \epsilon_0^{ 1-\alpha}.
\end{eqnarray}

Combining \eqref{sup1} and \eqref{c}, we obtain
  \begin{equation}\label{g}
|\psi_1|_\alpha \leq  C_1\epsilon_0^{ 1-\alpha} 
 \end{equation}
 for some constant $C_1$. 
 
 It remains to estimate $S_2^\epsilon$.
 First, let us estimate the oscilation of $\psi_2(z)=\mathbbm{1}_{B(y,r)}(z)$ over the set $B(x,\epsilon)$ when $r\leq\epsilon$:
 \begin{eqnarray*}osc(\psi_2,B(x,\epsilon))&=& \underset{ z\in B(x,\epsilon)\cap X}{ \textrm{ess-sup} } \mathbbm{1}_{B(y,r)}(z)- \underset{ \tilde{z}\in B(x,\epsilon)\cap X}{ \textrm{ess-inf} } \mathbbm{1}_{B(y,r)}(\tilde{z})\\
&\leq &\mathbbm{1}_{B(y,r+\epsilon)}(x).
\end{eqnarray*}
Thus,
\begin{equation}\label{f}
S_2^\epsilon:=\epsilon^{-\alpha}\int_{ X_\epsilon}osc(\psi_2,B_\epsilon(x))dx
\leq C_0\epsilon^{-\alpha}(r+\epsilon)^N
\leq 2^N \epsilon_0^{N-\alpha}.
\end{equation}

When $r>\epsilon$, we have
\begin{eqnarray*}
osc(\psi_2,B(x,\epsilon))\leq \mathbbm{1}_{B(y,r+\epsilon)\setminus B(y,r-\epsilon)}(x).
\end{eqnarray*}
Using the same ideas as in the estimation of  $\|\psi_1\|_\alpha$  and the last inequality, we obtain
\begin{equation}
S_2^\epsilon
\leq2^N \epsilon_0^{1-\alpha}Leb(X_{\epsilon_0}). \label{sup2-psi2}
\end{equation}

Thus, \eqref{f} and \eqref{sup2-psi2} give us
 \begin{equation}\label{psi2}
|\psi_2|_\alpha \leq  C_2\epsilon_0^{ 1-\alpha} 
 \end{equation}
for some constant $C_2$. 

From the inequalities \eqref{g} and \eqref{psi2}, we get (H2) and (H3).

 Finally, a straightforward calculation leads to ${C}_\mu=N$.


\medskip

\section{Proofs of the main results}\label{sec-proof}

\begin{proof}[Proof of Theorem \ref{thineq}]
For $\eps>0$, let us define 
\[k_n=\frac{1}{\underline{C}_\mu-\eps}(2\log n+\log\log n) \quad \text{and} \quad r_n=e^{-k_n}.\]
We also define
\[A_{ij}(y)=T^{-i}B(T^jy,e^{-k_n})\]
and
\[S_n(x,y)=\sum_{i,j=0,\dots,n-1}\mathbbm{1}_{A_{ij}(y)}(x).\]
Observe that
\begin{equation}\label{eqMnSn}
\left\{(x,y): m_n(x,y)< r_n\right\}= \left\{(x,y):S_n(x,y)>0\right\}.\end{equation}
Thus, we have
\[\mu\otimes\mu \left((x,y):m_n(x,y)< r_n\right)=\mu\otimes\mu \left((x,y): S_n(x,y)>0\right)=\mu\otimes\mu \left((x,y): S_n(x,y)\geq1\right).\]
Then, using Markov's inequality, we obtain
\begin{eqnarray*}
\mu\otimes\mu \left((x,y):m_n(x,y)< r_n\right)&\leq& \E(S_n)=\iint\sum_{i,j=0,\dots,n-1}\mathbbm{1}_{A_{ij}(y)}(x)d\mu\otimes\mu(x,y)\\
&=&\sum_{i,j=0,\dots,n-1}\int\left(\int\mathbbm{1}_{A_{ij}(y)}(x)d\mu(x)\right)d\mu(y)\\
&=&\sum_{i,j=0,\dots,n-1}\int\mu\left(B(T^jy,r_n)\right)d\mu(y),
\end{eqnarray*}
since $\mu$ is invariant.

Using again the invariance of $\mu$, we get
\[\mu\otimes\mu \left((x,y):m_n(x,y)< r_n\right)\leq n^2\int\mu(B(y,r_n)d\mu(y).\]
By the definition of the lower correlation dimension and the definition of $k_n$, for $n$ large enough, we have
\[\mu\otimes\mu \left((x,y):m_n(x,y)< r_n\right)\leq n^2r_n^{\underline{C}_\mu-\eps}=\frac{1}{\log n}.\]
Finally, choosing a subsequence $n_\ell=\lceil e^{\ell^2}\rceil$, we have 
\[\mu\otimes\mu \left((x,y):m_{n_\ell}(x,y)< r_{n_\ell}\right)\leq \frac{1}{\log n_\ell}\leq \frac{1}{\ell^2}.\]
Thus $\sum_\ell \mu\otimes\mu \left((x,y):m_{n_\ell}(x,y)< r_{n_\ell}\right)<+\infty$. By the Borel-Cantelli Lemma, for $\mu\otimes\mu$-almost every $(x,y)\in X\times X$, if $\ell$ is large enough then
\[m_{n_\ell}(x,y)\geq r_{n_\ell}\]
and
\[\frac{\log m_{n_\ell}(x,y)}{-\log n_\ell}\leq \frac{1}{\underline{C}_\mu-\eps}\left(2+\frac{\log\log n_\ell}{\log n_\ell}\right).\]
Finally, taking the limit superior in the previous equation and observing that  $(n_\ell)_\ell$ is increasing, $(m_n)_n$ is decreasing and $\underset{\ell\rightarrow+\infty}\lim\frac{\log n_\ell}{\log n_{\ell+1}}=1$, we have
\[ \underset{n\rightarrow+\infty}{\overline\lim}\frac{\log m_n(x,y)}{-\log n}=\underset{\ell\rightarrow+\infty}{\overline\lim}\frac{\log m_{n_\ell}(x,y)}{-\log n_\ell}\leq \frac{2}{\underline{C}_\mu-\eps}.\]
Then the theorem is proved since $\eps$ can be chosen arbitrarily small.
\end{proof}

Before proving Theorem \ref{dsup} we state a few facts in order to 
 simplify the calculations.
At first let us recall the notion of $(\lambda,r)$-grid partition.
\begin{definition}
Let $0<\lambda<1$ and $r>0$. A partition $\{Q_i\}_{i=1}^{\infty}$ of $X$ is called a $(\lambda,r)$-grid partition if there exists a sequence $\{x_i\}_{i=1}^{\infty}$ such that for any $i\in\N$
\[B(x_i,\lambda r)\subset Q_i\subset B(x_i, r).\]
\end{definition}
Now we prove a technical  lemma.
\begin{lemma}\label{lema1}
Under the hypotheses of Theorem \ref{dsup}, there exists a constant $K>0$ such that
\[\int_X\mu\left(B(y,r_n)\right)^2d\mu(y) \leq K\left(\int_X \mu \left(B(y,r_n)\right)d\mu(y)\right)^{3/2},\ \ \ \text{for } n\ \text{large enough}.\]
\end{lemma}
\begin{proof}

Since $X$ is a metric space, there exist $0<\lambda<\frac12$ and $R>0$ such that for any $0<r<R$ there exists a $(\lambda,r)$-grid partition (see Proposition 2.1 in \cite{GY}).

Let us choose $n$ large enough so that $r_n<\min\{R,r_0/2\}$ ($r_0$ as in Definition~\ref{deftight}). Let $\{Q_i\}_{i=1}^{\infty}$ be a $(\lambda,\frac{r_n}{2})$-grid partition and $\{x_i\}_{i=1}^\infty$ be such that $$B\left(x_i,\lambda \frac{r_n}{2}\right)\subset Q_i
\subset B\left(x_i, \frac{r_n}{2}\right).$$
Then we have
\begin{eqnarray}\label{a1}
\int_X\mu\left(B(y,r_n)\right)^2d\mu(y)&=&\sum_{i}\int_{Q_i}\mu\left(B(y,r_n)\right)^2d\mu(y).
\end{eqnarray}

 Now, fix a ball $B(x_i, {2r_n})$ and consider the set 
 $$D_i=\{x_j: Q_j \cap B(x_i,2r_n) \neq \emptyset\}.$$
 Since the space is tight, one can conclude that (see the proof of Theorem 4.1 in \cite{GY} ) there exists a constant $K_0$ depending only on $N_0$ such that the cardinality card($D_i$)$\leq K_0.$ Therefore
$$\bigcup_{y \in Q_i}B(y,r_n)\subset B(x_i,2r_n) \subset\bigcup_{j=1}^{K_0}Q_{i,j},$$
where $Q_{i,j}$ are elements of the partition.

By \eqref{a1} we have
\begin{eqnarray*}
& &\int_X\mu\left(B(y,r_n)\right)^2d\mu(y)
\leq\sum_{i}\int_{Q_i}\left(\sum_{j=1}^{K_0}\mu\left(Q_{i,j}\right)\right)^2d\mu(y)\\
&=&\sum_{i}\mu(Q_i)\left(\sum_{j=1}^{K_0}\mu\left(Q_{i,j}\right)\right)^2
\leq\sum_{i}\left(\sum_{j=1}^{K_0}\mu\left(Q_{i,j}\right)\right)^3
\leq K_0^2\sum_{i}\sum_{j=1}^{K_0}\mu\left(Q_{i,j}\right)^3,
\end{eqnarray*}
where the last inequality is deduced from Jensen's inequality.
Now, since the elements $Q_{i,j}$ cannot participate in more than $K_0$ different sums (one can see the arguments leading to (12) in \cite{GY}) and  since $x\mapsto x^{2/3}$ is a countably subbadditive function, we have
\begin{eqnarray*}
& &\int_X\mu\left(B(y,r_n)\right)^2d\mu(y)
\leq K_0^3\sum_{i}\mu\left(Q_{i}\right)^3\\
&\leq& K_0^3\left(\sum_{i}\mu\left(Q_{i})\right)^2\right)^{3/2}
=K_0^3\left(\sum_{i}\int_{Q_i}\mu\left(Q_i\right)d\mu(y)\right)^{3/2}.
\end{eqnarray*}
Finally, note that for any $y\in Q_i$, we have $Q_i\subset B(y,r_n)$. Thus
\begin{eqnarray*}
\int_X\mu\left(B(y,r_n)\right)^2d\mu(y)
\leq K_0^3\left(\sum_{i}\int_{Q_i}\mu\left(B(y,r_n\right)d\mu(y)\right)^{3/2},
\end{eqnarray*}
 and the result follows with $K=K_0^3$. 
\end{proof}
We are now ready to prove Theorem \ref{dsup}.
\begin{proof}[Proof of Theorem \ref{dsup}]
Without loss of generality, we will assume in the proof that $\theta_n=e^{-n}$.

For $\eps>0$, let us define 
\[k_n=\frac{1}{\overline{C}_\mu+\eps}(2\log n+b\log\log n)  \quad \text{and} \quad r_n=e^{-k_n}.\]
Using the same notation as in the proof of Theorem \ref{thineq}, we recall that
\begin{equation}\label{eqexp}\E(S_n)=n^2\int\mu(B(y,r_n))d\mu(y).\end{equation}
Moreover, using \eqref{eqMnSn} and Chebyshev's inequality, we obtain
\begin{align}
&\mu\otimes\mu \left((x,y):m_n(x,y)\geq r_n\right)\leq\mu\otimes\mu \left((x,y):S_n(x,y)=0\right) \nonumber \\
\leq&\mu\otimes\mu \left((x,y):|S_n(x,y)-\E(S_n)|\geq |\E(S_n)|\right)
\leq\frac{\var(S_n)}{\E(S_n)^2}. \label{eqmnvarsn}
\end{align}
Thus, we need to control the variance of $S_n$. First of all, we have
\begin{align*}
\var(S_n)
=&\sum_{1\leq i,i',j,j'\leq n}cov(\mathbbm{1}_{A_{ij}},\mathbbm{1}_{A_{i'j'}})\\
=&\sum_{1\leq i,i',j,j'\leq n}\iint \mathbbm{1}_{A_{ij}}\mathbbm{1}_{A_{i'j'}}-\iint\mathbbm{1}_{A_{ij}}\iint\mathbbm{1}_{A_{i'j'}}\\
=&\sum_{1\leq i,i',j,j'\leq n}\iint \mathbbm{1}_{B(T^jy,r_n)}(T^ix)\mathbbm{1}_{B(T^{j'}y,r_n)}(T^{i'}x)-n^4\left(\int\mu(B(y,r_n)d\mu(y)\right)^2.
\end{align*}
Let $g=g(n)=\log(n^{4+4\xi/(\overline{C}_\mu+\eps)})$. We will split the last sum into the   following four parts:
\begin{align*}
\sum_{1\leq i,i',j,j'\leq n}=&\sum_{|i-i'|>g,|j-j'|> g}+\sum_{|i-i'|>g,|j-j'|\leq g}+\sum_{|i-i'|\leq g,|j-j'|> g}+\sum_{|i-i'|\leq g,|j-j'|\leq g}\\
=:& I+II+III+IV.
\end{align*}

At first we observe that if $|i-i'|>g$, then by (H1) and (H3),
\begin{eqnarray}
& &\iint \mathbbm{1}_{B(T^jy,r_n)}(T^{i-i'}x)\mathbbm{1}_{B(T^{j'}y,r_n)}(x)d\mu(x)d\mu(y)\nonumber\\
&\leq&\int\left(\int\mathbbm{1}_{B(T^{j}y,r_n)}(x)d\mu(x)\int\mathbbm{1}_{B(T^{j'}y,r_n)}(x)d\mu(x)\right.\nonumber\\
& &+\left.\theta_g \cdot\|\mathbbm{1}_{B(T^{j}y,r_n)}\|_\mathcal{C} \cdot \|\mathbbm{1}_{B(T^{j'}y,r_n)}\|_\mathcal{C}\right)d\mu(y)\nonumber\\
&\leq&c^2r_n^{-2\xi}\theta_g+\int\mu\left(B(T^{j}y,r_n)\right)\mu\left(B(T^{j'}y,r_n)\right)d\mu(y).\label{eqmix1}
\end{eqnarray}
Therefore
\begin{align*}
I+II\leq& n^4c^2r_n^{-2\xi}\theta_g+n^2\sum_{|j-j'|> g}\int\mu\left(B(T^{j}y,r_n)\right)\mu\left(B(T^{j'}y,r_n)\right)d\mu(y)\\
&+n^2\sum_{|j-j'|\leq g}\int\mu\left(B(T^{j}y,r_n)\right)\mu\left(B(T^{j'}y,r_n)\right)d\mu(y).
\end{align*}

Now, in the case where $|j-j'|>g$, we use (H1) and (H2) to get
\begin{eqnarray}
& &\int\mu\left(B(T^{j}y,r_n)\right)\mu\left(B(T^{j'}y,r_n)\right)d\mu(y)\nonumber\\
&\leq&\left(\int\mu\left(B(y,r_n)\right)d\mu(y)\right)^2+\theta_g \cdot \|\mu\left(B(\cdot,r_n)\right)\|_\mathcal{C}\cdot \|\mu\left(B(\cdot,r_n)\right)\|_\mathcal{C}\nonumber\\
&\leq&\left(\int\mu\left(B(y,r_n)\right)d\mu(y)\right)^2+c^2 r_n^{-2\xi}\theta_g.\label{eqmix2}
\end{eqnarray}
Otherwise
we use H\"older's inequality and the invariance of the measure to obtain
\begin{eqnarray}
& &\int\mu\left(B(T^{j}y,r_n)\right)\mu\left(B(T^{j'}y,r_n)\right)d\mu(y)\nonumber\\
&\leq&\left(\int\mu\left(B(T^{j}y,r_n)\right)^2d\mu(y)\right)^{1/2}\left(\int\mu\left(B(T^{j'}y,r_n)\right)^2d\mu(y)\right)^{1/2}\nonumber\\
&=&\int\mu\left(B(y,r_n)\right)^2d\mu(y).\label{eqhold}
\end{eqnarray}
So the first two terms can be estimated as below
\begin{align}\label{c1}
I+II\leq& 2n^4c^2r_n^{-2\xi}\theta_g+n^4\left(\int\mu\left(B(y,r_n)\right)d\mu(y)\right)^2\nonumber \\
&+2n^3{g}\int\mu\left(B(y,r_n)\right)^2d\mu(y).
\end{align}

The third term can be treated exactly as the second one using the following symmetry on $x$ and $y$:
\begin{eqnarray*} 
\iint \mathbbm{1}_{B(T^jy,r_n)}(T^ix)\mathbbm{1}_{B(T^{j'}y,r_n)}(T^{i'}x)d\mu(x)d\mu(y)\\=\iint \mathbbm{1}_{B(T^ix,r_n)}(T^jy)\mathbbm{1}_{B(T^{i'}x,r_n)}(T^{j'}y)d\mu(y)d\mu(x).
\end{eqnarray*}
Finally, for the last term we use the boundedness of the indicator function and the invariance of the measure to obtain
\begin{eqnarray}
& &\iint \mathbbm{1}_{B(T^jy,r_n)}(T^ix)\mathbbm{1}_{B(T^{j'}y,r_n)}(T^{i'}x)d\mu(x)d\mu(y)\nonumber\\
&\leq&\iint \mathbbm{1}_{B(T^jy,r_n)}(T^ix)d\mu(x)d\mu(y)\nonumber\\
&\leq&\int\mu\left(B(y,r_n)\right)d\mu(y).\label{eqdiag}
\end{eqnarray}

Therefore,
\begin{align}\label{d}
III+IV\leq& n^4c^2r_n^{-2\xi}\theta_g+2  g n^3\int\mu\left(B(y,r_n)\right)^2d\mu(y)\nonumber \\
&+4n^2 g^2\int\mu\left(B(y,r_n)\right)d\mu(y).
\end{align}

Combining together the estimates  \eqref{eqexp}, \eqref{eqmnvarsn}, \eqref{c1} and \eqref{d}, 
we obtain
\begin{eqnarray*} 
\frac{\var(S_n)}{(\E(S_n))^2}&\leq& \frac{3n^4c^2r_n^{-2\xi}\theta_g+4n^2 g^2\int\mu\left(B(y,r_n)\right)d\mu(y)}{\left(n^2\int\mu(B(y,r_n)d\mu(y)\right)^2}\\
& &+\frac{4n^3 g\int\mu\left(B(y,r_n)\right)^2d\mu(y)}{\left(n^2\int\mu(B(y,r_n)d\mu(y)\right)^2}.
\end{eqnarray*}

To estimate the first term, we use the information on the decay of correlations (H1) and the choice of $g$ to obtain
\begin{eqnarray}
 \frac{3n^4c^2r_n^{-2\xi}\theta_g}{\left(n^2\int\mu(B(y,r_n))d\mu(y)\right)^2}
&\leq&3c^2\theta_gr_n^{-2(\overline{C}_\mu+\eps)}r_n^{-2\xi}\nonumber\\
&\leq&3c^2\theta_gn^{4}(\log n)^{2b}n^{4\xi/(\overline{C}_\mu+\eps)}(\log n)^{{2\xi b \over \overline{C}_\mu+\eps}}\nonumber\\
&\leq&3c^2(\log n)^{2b\left(1+{\xi \over\overline{C}_\mu+\eps}\right)}. \label{est1}
\end{eqnarray}

For the second term, we use again our choice of $g$ to get
\begin{eqnarray}
\frac{4n^2 g^2\int\mu\left(B(y,r_n)\right)d\mu(y)}{\left(n^2\int\mu(B(y,r_n))d\mu(y)\right)^2}
&\leq &4 g^2(\log n)^b
=4\left(4+\frac{4\xi}{\overline{C}_\mu+\eps}\right)^2(\log n)^{2+b}.\label{est3}
\end{eqnarray}

For the last estimate,  we use Lemma \ref{lema1} to obtain
\begin{eqnarray}
\frac{4n^3 g\int\mu\left(B(y,r_n)\right)^2d\mu(y)}{\left(n^2\int\mu(B(y,r_n))d\mu(y)\right)^2}
&\leq&\frac{4K  g }{n\left(\int\mu(B(y,r_n)d\mu(y)\right)^{1/2}}\nonumber\\
&\leq& 4K\frac{ g}{n}r_n^{-\frac{\overline{C}_\mu+\eps}{2}}\nonumber\\
&=&4K g\cdot(\log n)^{{b\over 2}}\nonumber \\
&=&4K\left(4+{4\xi \over \overline{C}_\mu+\eps}\right)(\log n)^{1+{b\over 2}}.\label{est2}
\end{eqnarray}

Now, taking  $b<-2$, and combining \eqref{eqmnvarsn}, \eqref{est1}, \eqref{est3} and \eqref{est2}, we have 
\[\mu\otimes\mu \left((x,y),m_n(x,y)\geq r_n\right)\leq\frac{\var(S_n)}{(\E(S_n))^2}\leq K_1(\log n)^{1+{b\over 2}},\]
for large enough $n$, and for some constant $K_1$.

Thus, choosing $b<-4$ and the subsequence $n_\ell=\lceil e^{\ell^2}\rceil$, 
we can apply the Borel-Cantelli Lemma as in the proof of Theorem \ref{thineq}. Moreover, since by Theorem \ref{thineq}  $m_n(x,y)>0$ a.e., we can consider the quantity $\log m_n(x,y)$ and we obtain
\[ \underset{n\rightarrow+\infty}{\underline\lim}\frac{\log m_n(x,y)}{-\log n}=\underset{\ell \rightarrow+\infty}{\underline\lim}\frac{\log m_{n_\ell}(x,y)}{-\log n_\ell}\geq \frac{2}{\overline{C}_\mu+\eps}.\]
Then the theorem is proved since $\eps$ can be chosen arbitrarily small.
\end{proof}

Now we explain how to modify the proof of Theorem \ref{dsup} to prove Theorem \ref{thprinc2}.
\begin{proof}[Proof of Theorem \ref{thprinc2}]
When our Banach space $\mathcal{C}$ is the space of H\"older functions, (H3) cannot be satisfied since the characteristic functions are not continuous. Thus the only difference in the proof will be in \eqref{eqmix1} and \eqref{eqmix2} where we use the mixing hypothesis.

First, one can easily adapt \eqref{eqmix1} in this setting approximating the characteristic functions by Lipschitz functions exactly as in the proof of Lemma 9 in \cite{rapid}.

Then, to obtain \eqref{eqmix2} we just need to prove that if (HA) is satisfied then the function $x\mapsto \mu(B(x,r))$ is H\"older, i.e. (H2) is satisfied.

In fact, let $x,y\in X$ and $0<r<r_0$. If $\|x-y\|<r $ then by (HA),
\begin{eqnarray*}
\| \mu(B(x,r))- \mu(B(y,r))\|&\leq&  \mu\left(B(x,r+\|x-y\|)\backslash B(x,r-\|x-y\|)\right)\\
&\leq&r^{-\xi}\|x-y\|^{\beta}.
\end{eqnarray*}
If $\|x-y\|\geq r $ then
\begin{eqnarray*}
\| \mu(B(x,r))- \mu(B(y,r))\|\leq 2
\leq\frac{2}{r}\|x-y\|.
\end{eqnarray*}
Thus, the function $x\mapsto \mu(B(x,r))$ is H\"older, one can applied (H1) and (H2) to obtain \eqref{eqmix2} and the theorem is proved.
\end{proof}


\medskip

\section{Proof of the symbolic case}\label{sec-discrete}

The proof of the first part of Theorem~\ref{seqmat} is a simple adaptation of the proof of Theorem \ref{thineq} and a simpler version was also proved in \cite{AW}. 
To do this adaptation, one must substitute $-\log m_n(x,y)$ by $M_n(x,y)$ and the balls $B(x, e^{-k})$ must be substituted by cylinders $C_k(x)$.

We will just focus on the second part of the theorem and explain the main differences with Theorem \ref{thprinc}. We will assume that the system is $\alpha$-mixing with an exponential decay, the $\psi$-mixing case can be easily deduced using the same ideas.
\begin{proof}[Proof of Theorem~\ref{seqmat}]
For $\eps>0$, let us define 
\[k_n=\frac{1}{\overline{H}_2+\eps}(2\log n+b\log\log n).\]
We also define
\[A_{ij}(y)=\sigma^{-i}C_{k_n}(\sigma^jy)\]
and
\[S_n(x,y)=\sum_{i,j=1,\dots,n}\mathbbm{1}_{A_{ij}(y)}(x).\]
We observe that
\begin{equation}\label{expsh}\E(S_n)=n^2\sum_{C_{k_n}} \P(C_{k_n})^2,\end{equation}
where the sum is taken over all the cylinders of size $k_n$.

Following the lines of the proof of Theorem \ref{thprinc}, we have
\begin{equation}\label{eqvarsh}
\P\otimes\P \left((x,y): M_n(x,y)\leq k_n\right)\leq\frac{\var(S_n)}{\E(S_n)^2}.\end{equation}
Again, we will estimate the variance dividing the sum of $\var(S_n)$ into $4$ terms. Let $g=\log(n^4).$
For $i-i'>g+k_n$, we have the equivalent of equation \eqref{eqmix1}:
\begin{eqnarray}
& &\iint \mathbbm{1}_{C_{k_n}(\sigma^jy)}(\sigma^ix)\mathbbm{1}_{C_{k_n}(\sigma^{j'}y)}(\sigma^{i'}x)d\P(x)d\P(y)\nonumber\\
&\leq&\alpha(g)+\int\P\left(C_{k_n}(\sigma^{j}y)\right)\P\left(C_{k_n}(\sigma^{j'}y)\right)d\P(y).\nonumber
\end{eqnarray}
If, moreover, $j-j'>g+k_n$, we have the equivalent of \eqref{eqmix2}, that is
\begin{eqnarray}
& &\int\P\left(C_{k_n}(\sigma^{j}y)\right)\P\left(C_{k_n}(\sigma^{j'}y)\right)d\P(y)\nonumber\\
 &=&\int\P\left(C_{k_n}(\sigma^{j-j'}y)\right)\P\left(C_{k_n}(y)\right)d\P(y)\nonumber\\
 &=&\sum_{C_{k_n},C'_{k_n}}\P(C_{k_n})\P(C'_{k_n})\P\left(C_{k_n}\cap \sigma^{-(j-j')}C'_{k_n}\right)\nonumber\\
 &\leq&\sum_{C_{k_n},C'_{k_n}}\P(C_{k_n})\P(C'_{k_n})\left(\P(C_{k_n})\P(C'_{k_n})+\alpha(g)\right)\nonumber\\
 &\leq&\alpha(g)+\left(\sum_{C_{k_n}}\P(C_{k_n})^2\right)^2.\nonumber
\end{eqnarray}
However, if $j-j'\leq g+k_n$, we obtain the following inequality, equivalent of \eqref{eqhold}:
\begin{equation*}
\int\P\left(C_{k_n}(\sigma^{j}y)\right)\P\left(C_{k_n}(\sigma^{j'}y)\right)d\P(y)\leq
\sum\P(C_{k_n})^3.\end{equation*}
As in the proof of Lemma \ref{lema1}, using the subbaditivity of $x\mapsto x^{2/3}$, we have
\begin{equation*}
\sum\P(C_{k_n})^3\leq\left(\sum\P(C_{k_n})^2\right)^{3/2}.\end{equation*}
Finally, when $|i-i'|\leq g+k_n$ and $|j-j'|\leq g+k_n$, we have the equivalent of \eqref{eqdiag}:
\begin{equation*}
\iint \mathbbm{1}_{C_{k_n}(\sigma^jy)}(\sigma^ix)\mathbbm{1}_{C_{k_n}(\sigma^{j'}y)}(\sigma^{i'}x)d\P(x)d\P(y)\leq\sum\P(C_{k_n})^2.
\end{equation*}

Then, one can gather these estimates to obtain
\begin{align*}
&\P\otimes\P \left((x,y):M_n(x,y)\leq k_n\right)\\
\leq &\frac{2n^4\alpha(g)+2n^3(g+k_n)\left(\sum\P(C_{k_n})^2\right)^{3/2}+n^2(g+k_n)^2\sum\P(C_{k_n})^2}{\left(n^2\sum \P(C_{k_n})^2\right)^2}.
\end{align*}
Thus, for $b<-2$,
\[
\P\otimes\P \left((x,y),M_n(x,y)\leq k_n\right)=\mathcal{O}((\log n)^{1+{b\over 2}}).\]
To conclude the proof, we use the Borel-Cantelli Lemma, exactly as in the proof of Theorem \ref{thprinc}.
\end{proof}

\medskip

\subsection*{Acknowledgements}The authors would like to thank Rodrigo Lambert and Mike Todd for various comments on a first draft of this article.

\end{document}